\documentclass[a4]{amsart}
\usepackage[all]{xy}
\usepackage{amsmath,amsfonts,amssymb,amsthm,a4,hyperref}
\usepackage{mathrsfs}

\usepackage[margin=3.65cm]{geometry}
\usepackage{hyperref}
\usepackage{cleveref}
\usepackage{setspace}
\usepackage{enumerate}
\usepackage{verbatim}
\usepackage{color}

\newtheorem{theorem}{\bf Theorem}[section]
\newtheorem{definition}[theorem]{\bf Definition}
\newtheorem{lemma}[theorem]{\bf Lemma}
\newtheorem{proposition}[theorem]{\bf Proposition}
\newtheorem{corollary}[theorem]{\bf Corollary}
\newtheorem{remark}[theorem]{\bf Remark}
\newtheorem{example}[theorem]{\bf Example}
\numberwithin{equation}{section}

\newcommand{\la}{\langle}
\newcommand{\ra}{\rangle}

\newcommand{\ot}{\otimes}

\newcommand{\C}{\mathbb{C}}

\newcommand{\Z}{\mathbb{Z}}

\newcommand{\mcal}{\mathcal}
\newcommand{\mA}{\mathfrak{A}}

\begin{document}

\title[Possible values of the interior angle]{On possible values of
  the interior angle between intermediate subalgebras}

\author[V P Gupta]{Ved Prakash Gupta}
\author[D Sharma]{Deepika Sharma*}

\address{School of Physical Sciences, Jawaharlal Nehru University, New Delhi, INDIA}

\email{vedgupta@mail.jnu.ac.in}
\email{sharmadeepikaq@gmail.com}


\date{}

\begin{abstract}
We show that all values in the interval $[0,\frac{\pi}{2}]$ can be
attained as the interior angle  between
intermediate subalgebras (as introduced in \cite{BG2}) of a certain inclusion of simple unital
$C^*$-algebras. We also calculate the interior  angle between
intermediate crossed product subalgebras of any inclusion of crossed
product algebras corresponding  to any action  of a countable discrete
group and its subgroups on a  unital $C^*$-algebra.
  \end{abstract}

\maketitle

\section{Introduction}
In any category, in order to classify its objects, the analysis of the
relative positions of the subobjects of an object has proved to be a
very rewarding approach. In the same vein, in the category of operator
algebras, a great deal of work has been done by some eminent
mathematicians - see, for instance, \cite{BDLR, BG2, Izumi, IW, JS,
  Wat} and the references therein. The theory of subfactors and, more
generally, the theory of inclusions of (simple) $C^*$-algebras are two
limelights of this aspect.

In this article, our focus lies only on unital $C^*$-algebras and their
subalgebras. Over the years, various significant tools and theories
have been developed to understand the relative positions of
subalgebras of a given unital $C^*$-algebra. Among them, Watatani's
notions of {\em finite-index conditional expectations} and {\em
  $C^*$-basic construction} with respect to a finite-index conditional
expectation (\cite{Wat}) have proved to be fundamental in the
development of the theory of inclusions of $C^*$-algebras - see
\cite{Wat, MK, IW, Izumi, BG2}. Based on these two notions, and
motivated by \cite{BDLR}, very recently, Bakshi and the first named
author, in \cite{BG2}, introduced the notions of {\em interior and
  exterior angles} between intermediate $C^*$-subalgebras of a given
inclusion $B \subset A$ of unital $C^*$-algebras with a finite-index
conditional expectation. As an application of the notion of 
interior angle, the authors in \cite{BG2} were  able to improve
Longo's upper bound for the cardinality of the lattice of intermediate
$C^*$-subalgebras of any irreducible inclusion of simple unital
$C^*$-algebras.

Apart from the above mentioned quantitative application of the notion
of interior angle, we expect some significant qualitative consequences
too to be visible soon. In this direction, it is then quite natural to
first ask whether one can make some concrete calculations of these
angles and the possible values that they can attain. This article
essentially answers these questions to a certain level of
satisfaction. Being precise, through some elementary calculations, we
are able to show that all values in the interval $[0,\frac{\pi}{2}]$
are attained as the interior angles between intermediate subalgebras
of a certain inclusion of simple unital $C^*$-algebras. Further,
motivated by \cite{BG1}, we also calculate the interior angle between
intermediate crossed product subalgebras of any inclusion of crossed
product algebras corresponding to any action of a countable discrete
group and its subgroups on a given unital $C^*$-algebra.

The article is organized as follows:

After the introduction, we have a relatively longer section on
preliminaries wherein we recall and derive some basic nuances related
to finite-index conditional expectations and Watatani's $C^*$-basic
construction related to inclusions of unital $C^*$-algebras. This
discussion is fundamental to the formalism of the interior and
exterior angles, which we briefly recall in Section 3; and, in the same section, we also derive some
useful expressions related to them. Then, in
\Cref{possible-values}, we prove that for any $t \in
     [0,\frac{\pi}{2}]$, there exists a $2 \times 2$ unitary matrix
     $u$ such that the interior angle $\alpha(\Delta, u \Delta u^*) =
     t$ with respect to the canonical conditional expectation from
     $M_2(\C)$ onto $\C$, where $\Delta$ denotes the diagonal
     subalgebra of $M_2(\C)$; thereby, establishing that all values in
     the interval $[0, \frac{\pi}{2}]$ are attained as the interior
     angles between intermediate subalgebras. Finally, in
     \Cref{crossed-products}, as an application of some expressions
     derived in Section 3, given any quadruple of countable discrete
     groups $H \subsetneq K, L \subsetneq G$ with $[G:H]< \infty$ and
     with an action $\alpha$ of $G$ on a unital $C^*$-algebra $P$, we
     derive  an expression for the interior angle between the (reduced as
     well as universal) intermediate crossed product subalgebras $P
     \rtimes K$ and $P \rtimes L$ of the inclusion $P \rtimes H
     \subset P \rtimes G$. 
\section{Preliminaries}

\subsection{Watatani's index and basic construction} \( \)
In this subsection, we first recall Watatani's notions of finite-index
conditional expectations and the $C^*$-basic constructions with
respect to such conditional expectations; and then, we touch upon some
generalities related to intermediate $C^*$-subalgebras.

\subsubsection{Finite-index conditional expectations}

Recall that, for an inclusion $B\subset A$ of unital $C^*$-algebras, a
conditional expectation $E: A\to B$ is said to have \textit{finite
  index} if there exists a finite set
$\{\lambda_1,\ldots,\lambda_n\}\subset A$ such that
\[
x=\sum_{i=1}^n E(x\lambda_i)\lambda^*_i=\sum_{i=1}^n
\lambda_iE(\lambda^*_ix)
\]
for every $x\in A$ - see \cite{Wat, Izumi, IW}. Such a set
$\{\lambda_1,\ldots,\lambda_n\}$ is called a \textit{quasi-basis} for
$E$ and the Watatani index of $E$ is defined as
 \[
 \mathrm{Ind} (E)= \sum_{i=1}^n \lambda_i\lambda^*_i.
 \]
It is known that $\mathrm{Ind}(E)$ is a positive invertible element of
$\mathcal{Z}(A)$ and is independent of the quasi-basis $\{\lambda_i\}$
- see \cite[$\S 2$]{Wat}. Also, $E$ is faithful, $E(1_A) = 1_B$ and
$\mathrm{Ind}(E)\geq 1$.

\begin{remark}\label{index-facts}
  \begin{enumerate}
  \item If $B \subset C \subset A$ are inclusions of unital
    $C^*$-algebras with $1_A \in B$ and $E: A \to B$, $F: A \to C$ and
    $G: C \to B$ are faithful conditional expectations satisfying $E =
    G \circ F$, then $E$ has finite index if and only if $F$ and $G$
    have finite index - see \cite[Proposition 3.5]{MK}.
\item For an inclusion $B \subset A$, in general, if $E, E': A \to B$ are
two conditional expectations, one may be of finite index and the other
may fail - see \cite[Example 2.10.1]{Wat}.

Interestingly, if there exists a finite index conditional expectation from $A$ onto
$B$, then all faithful conditional expectations from $A$ onto $B$ are
of finite index if the centralizer of $B$ in $A$, i.e.,
$\mathcal{C}_A(B):=\{x \in A: xb=bx\ \text{for all}\ b \in B\}$, is
finite dimensional - see \cite[Propositions 2.10.2]{Wat}.

Thus, when $\mathcal{C}_A(B)$ is finite dimensional, one can roughly
say that the property of `finite index' is an intrinsic property of
the inclusion $B \subset A$ and not of a conditional expectation from
$A$ onto $B$.
\item There exist finite-index conditional expectations even when the
  corresponding centralizers are not finite dimensional. For instance, 
  see \cite[Example 2.6.7]{Wat}.

  Let $A=C(X)$ and $B:=A^\alpha$, where $X$ is an infinite compact
  Hausdorff space and $\alpha$ is a free action of a finite group $G$
  on $A$. Define $E : A \to B$ by
 \[ E (f)= \frac{\sum_{g} \alpha_{g} (f)}{|G|}, \ f\in A.
 \]
Then, $E$ has finite-index and $\mathrm{Ind}(E) = |G|$ - see
\cite[Proposition 2.8.1]{Wat} - whereas $\mathcal{C}_{A}(B)$ is
infinite dimensional as $A= C(X)$ is a commutative $C^*$-algebra.
\end{enumerate}
\end{remark}

\subsubsection{Watatani's $C^*$-basic construction} 
Let $B\subset A$ be an inclusion of unital $C^*$-algebras with common
unit and suppose $E: A \to B$ is a faithful conditional expectation.
Let $A_1$ denote the Watatani's $C^*$-basic construction of the
inclusion $B \subset A$ with respect to the conditional expectation
$E$, i.e., in short, one essentially shows the following:
\begin{enumerate}
  \item $A$ is a pre-Hilbert $B$-module with respect to the $B$-valued inner product given by
\[
\la a, a'\ra_B:=E(a^*a')\text{ for } a, a'\in A;
\]
and, if  $\mathfrak{A}$ denotes the Hilbert $B$-module completion of $A$, then
\item the space of adjointable maps on $\mA$, denoted by
  $\mathcal{L}_B(\mA)$, is a unital $C^*$-algebra (with the usual
  operator norm) and $A$ embeds in it as a unital $C^*$-subalgebra
  (and, by a slight abuse of notation, we identify $A$ with its image
  in $\mathcal{L}_B(\mA)$);
\item there exists a projection $e_B\in \mathcal{L}_B(\mA)$ (called the
  {\em Jones projection} associated to $E$) such that $e_Bae_B=
  E(a)e_B$ for all $a \in A$ (it is standard to denote $e_B$ by $e_1$
  as well);  and 
\item one considers $A_1 := \overline{\mathrm{span}}\{xe_By: x,y \in
  A\} \subseteq \mathcal{L}_B(\mA)$, which turns out to be a 
  $C^*$-algebra (not always unital)  and is called the $C^*$-basic construction of the
  inclusion $B \subset A$.
 \end{enumerate}

The system $(A, B,E, e_B, A_1)$ has the following natural
 universal property.

 \begin{theorem}\cite[Proposition 2.2.11]{Wat}\label{universal-property}
Let $B \subset A$ be an inclusion of unital $C^*$-algebras with a
faithful conditional expectation $E: A \to B$. Suppose that $A$ acts
   faithfully on some Hilbert space $H$ and $e$ is a projection on $
   H$ satisfying $eae=E(a)e$ for all $a \in A$. If the linear map $B \ni b
   \mapsto be \in B(H)$ is injective, then there is a $*$-isomorphism
   $\theta: A_1 \to \overline{AeA} \subset B(H)$ such that
   $\theta(xe_By) = x ey$ for all $x, y \in A$.
    \end{theorem}

\begin{remark}\label{useful-obsns}
  If $E:A \to B$ has finite index with a quasi-basis $\{\lambda_i\}$, then
  \begin{enumerate}
  \item the two norms $\|\cdot\|_A$ and $\|\cdot\|$ are equivalent on $A$ (where
  $\|x\|_A:= \|E_B(x^*x)\|^{1/2}$)- see \cite{Wat} or the proof of \cite[Lemma
  2.11]{BG2}; in particular, $A$ itself is a Hilbert $B$-module;
\item $A_1$ is unital and is equal to $C^*(A, e_B)$ - see \cite[Proposition .1.5]{Wat};
  \item there exists a finite-index conditional expectation $E_1: A_1
  \to A$ (called the dual conditional expectation) with a quasi-basis
  $\{\lambda_i e_B (\mathrm{Ind}(E))^{1/2}\}$ which satisfies the
  equation
  \begin{equation}\label{dual-ce}
E_1(xe_By) =\mathrm{Ind}(E)^{-1}xy
  \end{equation}
  for all $x, y \in A$ and $\mathrm{Ind}(E_1) = \sum_i \lambda_i
  E(\mathrm{Ind}(E))e_B \lambda_i^*$; moreover,
     if $\mathrm{Ind}(E)\in B$, then $\mathrm{Ind}(E_1) =
      \mathrm{Ind}(E)$ -  see \cite[Propositions
    2.3.2 $\&$ 2.3.4]{Wat}; and,
\item if $F: A \to B$ is another finite-index conditional expectation
  and $C^*(A, f_B)$ denotes the corresponding $C^*$-basic
  construction, then there exists a $*$-isomorphism $\theta: A_1 \to
  C^*(A, f_B)$ such that $\theta(e_B) = f_B$ and $\theta(a)=a$ for all $a
  \in A$ - \cite[Proposition 2.10.11]{Wat}; and,
      \item $A_1 = \mathrm{span}\{xe_By : x, y \in A\}=:Ae_BA$ - see
 \cite[Lemma 2.2.2]{Wat}.
      \end{enumerate}
\end{remark}

\subsection{Intermediate $C^*$-subalgebras}

Throughout this subsection, we let $B \subset A$ be an inclusion of
unital $C^*$-algebras, $E: A \to B$ be a finite-index conditional
expectation with a quasi-basis $\{\lambda_i: 1 \leq i \leq n\}$,
$A_1:=Ae_BA\ (=C^*(A, e_B))$ denote the $C^*$-basic construction of $B
\subset A$ with respect to $E$ and $E_1: A_1 \to A$ denote the dual
conditional expectation.

As in \cite{IW}, let $\mathrm{IMS}(B, A, E)$ denote the set of
intermediate $C^*$-subalgebras $C$ between $B$ and $ A$ with a
 conditional expectation $F: A\to C$ satisfying the
 compatibility condition $E= E_{\restriction_C} \circ F$.
 \begin{remark}\label{IMS-facts}
  \begin{enumerate}
\item  If $C \in \mathrm{IMS}(B, A, E)$      with respect to two compatible conditional expectations $F, F': A \to C$,
    then $F = F'$ - see \cite[Page 3]{IW}.

  \item   If $C \in \mathrm{IMS}(B, A, E)$ with respect to the
      compatible conditional expectation $F: A \to C$, then $F$ is faithful (since $E$ is so) and, therefore, by
 \Cref{index-facts}(1), $F$ has finite index.

  \item It must be mentioned here that it was presumed (without mention) in \cite{BG2}
    that the compatible conditional expectation has finite index and
    was implicitly used while defining the notions of interior and exterior
    angles between intermediate subalgebras of an inclusion of unital
    $C^*$-algebras.
\item For $C \in \mathrm{IMS}(B, A, E)$ with respect to the compatible conditional
expectation $F: A \to C$, we observe that $A$ is a Hilbert $C$-module
(\Cref{useful-obsns}(2)); we let $e_C$ denote the corresponding Jones
projection in $\mathcal{L}_C(A)$ and $C_1$ denote the Watatani
basic construction of the inclusion $C \subset A$; thus, $C_1= C^*(A,
e_C) \subseteq \mathcal{L}_C(A)$.
  \end{enumerate}
 \end{remark}

\begin{remark}
 In general, if $Q \subset P$ is an inclusion of unital
  $C^*$-algebras with a finite-index conditional expectation $G: P \to
  Q$, then not every intermediate $C^*$-subalgebra $R$ of $Q \subset P$
  belongs to $\mathrm{IMS}(Q,P, G)$ - see \cite[Example 2.5]{IW}. In
  fact, the example given in \cite{IW} illustrates that there need not 
  exist even a single conditional expectation from $P$ onto $R$.
  \end{remark}
Izumi observed that the intermediate subalgebras of an inclusion of simple $C^*$-algebras have certain specific structures.
\begin{proposition}\cite{Izumi}\label{intermediate-structure}
  Let $B \subset A$ be an inclusion of unital $C^*$-algebras with a
  finite-index conditional expectation $E: A \to B$. If either $A$ or $B$ is simple then, 
 every $C$ in $\mathrm{IMS}(B, A, E)$ is a finite direct sum of simple closed two-sided ideals.
\end{proposition}
\begin{proof}
Let $C \in \mathrm{IMS(B,A,E)}$ with respect to the compatible
conditional expectation $F: A\to C$.  Then, by \Cref{IMS-facts}(2) and
\cite[Proposition 2.1.5]{Wat}, $F$ and $ E_{\restriction_C }$ satisfy the Pimsner-Popa
inequality. Further, since $A$ or $B$ is simple and unital, it
then follows from \cite[Theorem 3.3]{Izumi} that $C$ is a finite direct sum
of simple closed two-sided ideals
\end{proof}

The following useful observations will be needed ahead when we recall and derive some generalities related to the notions of interior and exterior angles.
\begin{proposition}\label{eC-e1}
  Let $B \subset A$ be an inclusion of unital $C^*$-algebras, $E: A
  \to B$ be a finite-index conditional expectation with a quasi-basis
  $\{\lambda_i: 1 \leq i \leq n\}$, $A_1$ denote the
  $C^*$-basic construction of $B \subset A$ with respect to $E$, $E_1:
  A_1 \to A$ denote the dual conditional expectation and $C \in
  \mathrm{IMS}(B, A, E)$ with respect to the compatible finite-index
  conditional expectation $F: A \to C$.  Then,
\begin{enumerate}
\item $\mathcal{L}_C(A) \subset \mathcal{L}_B(A)$;
\item $C_1 \subset A_1$, so that $e_C \in A_1$;
\item $e_C e_B= e_B=e_B e_C$; 
  \item $E_{\restriction_C}$ has finite index with a quasi-basis
    $\{F(\lambda_i)\}$ and $e_C = \sum{\mu_j e_B \mu_j^*}$ for any
    quasi-basis $\{\mu_j \}$ of the conditional expectation
    $E_{\restriction_C}$;
  \item $E_1(e_B) = \mathrm{Ind}(E)^{-1}\in \mcal{Z}(A)$; 
  \item $E_1(e_C) = \mathrm{Ind}(E)^{-1} \mathrm{Ind}(E_{\restriction_C})\in \mcal{Z}(C) $; and,
  \item in addition, if $\mathrm{Ind}(E_{\restriction_C})\in \mcal{Z}(A)$, then
    \begin{enumerate}
    \item $\mathrm{Ind}(E) = \mathrm{Ind}(F) \mathrm{Ind}(E_{\restriction_C})$; 
       \item ${E_1}_{\restriction_{C_1}} = F_1$, where $F_1$ denotes the dual conditional expectation of $F$; and
       \item $C_1 \in \mathrm{IMS}(A, A_1, E_1)$ with respect to the
         conditional expectation $G: A_1 \to C_1$ satisfying $G(xe_By)
         = \mathrm{Ind}(E_{\restriction_C})^{-1} xe_Cy$ for all $x,y \in A$ and has a quasi-basis
         $\{\lambda_i e_B \mathrm{Ind}(E_{\restriction_{C}})^{1/2} :
         1\leq i \leq n\}$.
    \end{enumerate}
       In particular,  we   then have
       \(
       E_1(e_C) = \mathrm{Ind}(F)^{-1}.
       \)
\end{enumerate}
\end{proposition}

\begin{proof}
  (1): Let $T \in \mathcal{L}_C(A)$ and $T^*$ denote its adjoint in
  $\mathcal{L}_C(A)$. Then, we see that
  \begin{align*}{c}
  \la T(x), y\ra_B = E(T(x)^*y)= (E_{\restriction_C}\circ F)(T(x)^* y) = E_{\restriction_C} (\la T(x), y \ra_C) 
  \\
  = E_{\restriction_C} (\la x, T^*(y) \ra_C) =  (E_{\restriction_C}\circ F)(x T^* (y)) = \la x, T^*(y) \ra_B
  \end{align*}
 for all $x, y \in A$. Hence, $T \in \mcal{L}_B(A)$.
 \smallskip

  Because of (1), (2) now follows on the lines \cite[Lemma
    4.2]{BG2}. \smallskip

 (3): Clearly, $e_C e_B = e_B$ (as $B\subset
C$). Next, we observe that
  \[
  e_Be_C(a)= e_B(F(a)) = E(F(a)) = (E_{\restriction_C} \circ F)(a) =
  E(a) =e_B(a)
  \]
  for all $a \in A$.     Thus, $e_B e_C = e_B$.
                \smallskip
                
 (4): That $E_{\restriction_C}$ has finite-index with
  quasi-basis $\{F(\lambda_i)\}$ follows from \cite[Page 3]{IW} (also
  see \cite[Proposition 1.7.2]{Wat}). Further, for any quasi-basis
  $\{\mu_j\}$ for $E_{\restriction_C}$, we have
   \begin{eqnarray*}
              	(\sum \mu_{j} e_B \mu_{j}^*)(a) & = & \sum \mu_{j} e_B \mu_{j}^*(a)\\
     & = & \sum \mu_{j} E(\mu_{j}^*(a))\\
    &  = & \sum \mu_{j} (E_{\restriction_C} \circ F)(\mu_{j}^*(a))\\
   &   = & \sum \mu_{j} E_{\restriction_C}(\mu_{j}^* F(a))\\
                &    = & F(a) \\
                & =  &              e_C(a)
\end{eqnarray*}     for all $a\in A$. Hence, $e_C = \sum \mu_j e_B \mu_j^*$.\\
           \smallskip

           (5): See \cite[ Proposition 2.3.2]{Wat}. \smallskip

           (6): For any quasi-basis
    $\{\mu_j \}$ for  the conditional
    expectation $E_{\restriction_C}$, we have $ e_C = \sum \mu_{j} e_B \mu_{j}^* $. Hence,
\begin{eqnarray*}
	E_1(e_C) & = & E_1(\sum \mu_{j} e_B \mu_{j}^*)\\
	& =&  \sum E_1(\mu_{j} e_B \mu_{j}^*) \\
        & = & \mathrm{Ind}(E)^{-1} \sum \mu_j \mu_j^*\\
        & =  & \mathrm{Ind}(E)^{-1} \mathrm{Ind}(E_{\restriction_C}). 
\end{eqnarray*}

(7a): Let $\{\mu_1, \mu_2,..., \mu_n\}$ be a quasi-basis for $
E_{\restriction_C}$ and $\{\gamma_1,\gamma_2,..., \gamma_m\}$
be a quasi-basis for $F$. Then, it is (known and can be) easily
seen that $\{\gamma_i \mu_j : 1\leq i\leq m , 1 \leq j \leq n\}$ is a
quasi-basis for $E$ - see also \cite[Proposition 1.7.1]{Wat}. Thus,
\begin{eqnarray*}
	\mathrm{Ind}(E) & = & \sum_{i,j} (\gamma_i \mu_j)(\gamma_i \mu_j)^*\\
	& = & \sum_{i} \gamma_i (\sum_{j} \mu_j \mu_j^*) \gamma_i^* \\
        & = & \mathrm{Ind}(E_{\restriction_C}) \mathrm{Ind}(F).
\end{eqnarray*}

(7b): We have $C_1=\mathrm{span}\{xe_C y: x, y \in A\}$.  Fix a
quasi-basis $\{\mu_j\}$ for $E_{\restriction_C}$. Then, for every pair
$x,y\in C$, we observe that
\begin{eqnarray*}
  E_1(xe_Cy) & = & E_1(x \sum_j \mu_j e_B \mu_j^* y)\\
  & = & \mathrm{Ind}(E)^{-1}
\sum_j x \mu_j \mu_j^* y \\
& = & \mathrm{Ind}(E)^{-1}  x\,
\mathrm{Ind}(E_{\restriction_C}) y = \mathrm{Ind}(F)^{-1}xy \\
& = &
F_1(xe_Cy),
\end{eqnarray*}
where the second last equality follows from (7a). Hence, $(E_1)_{\restriction_{C_1}} = F_1$. \smallskip

(7c): We have $A_1=\mathrm{span}\{xe_B y: x, y \in A\}$. Consider the linear  map $G: A_1 \to C_1$  given by
\[
G(\sum_ix_ie_By_i)= \mathrm{Ind}(E_{\restriction_{C}})^{-1}\sum_i x_ie_C y_i
\]
for $x_i, y_i \in A$, $i=1, \ldots, n$. 

We first assert that $G$ is a  conditional expectation of finite-index. 

Fix a quasi-basis $\{\mu_j\}$ for $E_{\restriction_C}$. Then,  for
any $x,y \in A$, by Item (4), we have
\begin{eqnarray*}
  G(xe_Cy) & = & G(x \sum_j \mu_j e_B \mu_j^* y)\\ & = &
  \mathrm{Ind}(E_{\restriction_{C}})^{-1} \sum_j x \mu_j e_C \mu_j^*
  y\\ & = & \mathrm{Ind}(E_{\restriction_{C}})^{-1} \sum_j x \mu_j
  \mu_j^* e_C y\qquad (\text{since } e_C \in C'\cap C_1) \\ & = &
  \mathrm{Ind}(E_{\restriction_{C}})^{-1} x
  \mathrm{Ind}(E_{\restriction_{C}}) e_C y\\ & = & xe_Cy. \qquad
  (\text{since } \mathrm{Ind}(E_{\restriction_{C}}) \in
  \mathcal{Z}(A))
\end{eqnarray*}
This implies that $G^2 = G$.  Further, for any $\sum_i x_i e_B y_i \in
A_1$, we have
\begin{eqnarray*}
 G\big((\sum_{i} x_i e_B y_i)^* (\sum_{i} x_i e_B y_i)\big)&=& G\big(\sum_{i,j} y_i^* E(x_i^* x_j)e_B y_j\big)\\
 &=& \mathrm{Ind}(E_{\restriction_{C}})^{-1} \sum_{i,j} y_i^* e_CE(x_i^* x_j) e_C y_j.
 \end{eqnarray*}
Then, taking $a_{i,j} := e_C E(x_i^*x_j) e_C\in C_1$, $i, j = 1, \ldots , n$, we  have
  \[
[a_{i,j}]= \mathrm{diag}(e_C, \ldots, e_C)[E(x_i^* x_j)]
\mathrm{diag}(e_C, \dots, e_C).
\]
By \cite[Lemma 3.1]{Tak1}, $[x_i^*x_j]$ is positive in $M_n(A)$ and
since $E:A \to B$ is completely positive, it follows that
$[E(x_i^*x_j)]$ is positive in $M_n(B)$. Hence, $[a_{i,j}]$ is positive
in $M_n(C_1)$. Thus, by \cite[Lemma 3.2]{Tak1}, it follows that
$\sum_{i,j} y_i^* e_CE(x_i^* x_j) e_C y_j \geq 0$ in $C_1$. Further,
since ${\mathrm{Ind}(E_{\restriction_C})}^{-1} \in \mathcal{Z}(A)\cap
  \mathcal{Z}(C)$ and $e_C \in C'\cap C_1$, it follows that
  ${\mathrm{Ind}(E_{\restriction_C})}^{-1}$ commutes with $\sum_{i,j} y_i^* e_CE(x_i^* x_j) e_C y_j$ and hence
  \[
 G\big((\sum_{i} x_i e_B y_i)^* (\sum_{i} x_i e_B y_i)\big) \geq 0.
 \]
Thus,  $G: A_1 \to C_1$ is positive and, therefore, it is a conditional expectation.

 Further, $ G : A_1 \to C_1$ has finite index with quasi basis $\{\lambda_i e_B
\mathrm{Ind}(E_{\restriction_{C}})^{1/2} : 1\leq i \leq n\}$ because,	for any $x, y \in A$, we have
	\begin{eqnarray*}
	  \lefteqn{\sum_{i} \lambda_i e_B\mathrm{Ind}(E_{\restriction_{C}})^{1/2}G\big((\mathrm{Ind}(E_{\restriction_{C}})^{1/2} e_B \lambda_i^* x e_B y\big)}\\
          & = & \sum_{i} \lambda_i e_B(\mathrm{Ind}(E_{\restriction_{C}})^{1/2})G\big((\mathrm{Ind}(E_{\restriction_{C}})^{1/2}) E(\lambda_i^* x) e_B y\big)\\
		&=& \sum_{i} \lambda_i e_B\mathrm{Ind}(E_{\restriction_{C}})^{1/2} \mathrm{Ind}(E_{\restriction_{C}})^{-1} \mathrm{Ind}(E_{\restriction_{C}})^{1/2} E(\lambda_i^* x) e_C y\\
		& =& xe_By. \qquad \text{(since $e_C \in C_1 \cap B'$ , $e_B e_C = e_B$ and $e_B \in A_1 \cap B'$)}
	\end{eqnarray*}

        Finally, since ${E_1}_{\restriction_{C_1}}= F_1$, we observe that
\begin{eqnarray*}
({E_1}_{\restriction_{C_1}} \circ G) (xe_By) & = &
  E_1(\mathrm{Ind}(E_{\restriction_{C}})^{-1} x e_C y )
  \\ & = &
F_1(\mathrm{Ind}(E_{\restriction_{C}})^{-1} x e_C y )
\\
& = & \mathrm{Ind}(E_{\restriction_{C}})^{-1} \mathrm{Ind}({F})^{-1} x y \\
& = &
\mathrm{Ind}({E})^{-1} x y\\
& = & E_1(xe_By)
\end{eqnarray*}
for all $x, y \in A$, where the second last equality follows from Item
(7a). Hence, $({E_1}_{\restriction_{C_1}} \circ G) = E_1$.

These show that $C_1 \in \mathrm{IMS}(A, A_1, E_1)$ with respect to the
finite-index conditional expectation $G: A_1 \to C_1$.  
\end{proof}

   Recall that for an inclusion $B \subset A$ of unital
   $C^*$-algebras, the normalizer of $B$ in $A$ is defined as
   \[
   \mathcal{N}_A(B):=\{ u \in \mathcal{U}(A): u B u^* = B\},
   \]
where $\mathcal{U}(A)$ denotes the group of unitaries in $A$; and, (as
already mentioned above) the centralizer
of $B$ in $A$ is defined as
   \[
   \mathcal{C}_A(B):=\{ a \in A: ab  =ba \text{ for all } b\in B\}.
   \]
   Clearly, $\mcal{U}(B)$ is a normal subgroup of $\mathcal{N}_A(B)$
   and $\mcal{C}_A(B)$ is a unital $C^*$-subalgebra of $A$, which is also
   denoted by $B'\cap A$ and is called the {\em relative commutant} of
   $B$ in $A$.\medskip

The following observation provides us with some easy examples of
elements in $\mathrm{IMS}(B,A,E)$.

\begin{lemma} \label{conj_F}
 Let $B \subset A$ be an inclusion of unital $C^*$-algebras with a
 finite-index conditional expectation $E : A \to B$. Let $C \in
 \mathrm{IMS}(B,A,E)$ with respect to the compatible finite-index
 conditional expectation $F :A \to C$ and $u\in \mathcal{U}(A)$. Then,
 \begin{enumerate}
 \item $F_u : A \to uCu^*$ given by $F_u= \mathrm{Ad}_{u} \circ F
   \circ \mathrm{Ad}_{u^*}$, i.e., \( F_u(a) = u F(u^*au)u^* \text{
     for } a \in A, \) is a finite-index conditional expectation with
   a quasi-basis $\{ u \eta_i u^* : 1 \leq i \leq n\}$, where $\{
   \eta_i : 1 \leq i \leq n\}$ is a quasi basis for $F$;
   $\mathrm{Ind}(F_u) = \mathrm{Ind}(F)$; and,
   \item in addition, if $u \in \mathcal{N}_A(B)$ and $E$ satisfies the tracial
     property, i.e., $E(xy) = E(yx)$ for all $x,y \in A$, then
      $uCu^*
     \in \mathrm{IMS}(B,A,E)$ with respect to $F_u$.
        \end{enumerate}
 \end{lemma}
\begin{proof}
(1) is a straight forward verification.\smallskip

  (2): Let $D:=u Cu^*$. Since $u \in \mathcal{N}_A(B)$, $B = uBu^*
  \subset u Cu^*$ and we have
 \begin{eqnarray*}
 	E_{\restriction_D}\circ F_u(a) & = & E_{\restriction_D}(u F(u^*au)u^*)\\
 	& = & E( u^* uF(u^*au)) \qquad (\text{by tracial property})\\
 	 	& = & E_{\restriction_C}\circ F(u^*au)\\
 	& = & E(u^*au)    \\
 	& = & E(uu^*a)  \qquad (\text{by tracial property again})\\
 	& = & E(a)
 \end{eqnarray*}
 for all $a \in A$. Hence, $uCu^* \in \mathrm{IMS}(B,A,E)$ with respect to $F_u$. 
\end{proof}

\begin{remark} Note that $ue_Cu^*$ is a  projection in $C_1$ (as $u\in A\subset C_1$) and, for each $x \in A$, we have
 \[
(ue_Cu^*) x (ue_Cu^*)= uF(u^*xu)e_Cu^*=F_u(x)ue_Cu^*.
 \]
So, it is quite tempting to think that maybe    
   the basic construction of $B \subset uCu^*$ is given by $(uCu^*)_1 = C_1$ (as the $C^*$-algebra) with Jones projection   $e_{uCu^*} = ue_Cu^*$. However, this is not the case.    
   
   For instance, if we let $A$, $B$, $C$, $E: A\to B$ and $F: A \to C$ be the  same as in \Cref{possible-values}, then taking  $u= \begin{bmatrix}
   	 \frac{1}{\sqrt{2}} & \frac{i}{\sqrt{2}} \\ \frac{i}{\sqrt{2}} &
   	\frac{1}{\sqrt{2}}
   \end{bmatrix}$, we observe that $ue_Cu^* = \begin{bmatrix}
   1/2 & 0 & -i/2 & 0 \\ 0 & 1/2 & 0 & i/2 \\ i/2 & 0 & 1/2 & 0 \\ 0 & -i/2 & 0 & 1/2
\end{bmatrix}$ whereas $e_{uCu^*} = \begin{bmatrix}
1/2 & 0 & 0 & 1/2 \\ 0 & 1/2 & -1/2 & 0\\ 0 & -1/2 & 1/2 & 0 \\ 1/2 & 0 & 0 & 1/2
\end{bmatrix}$ (using values of $e_C$ from \Cref{E-F-finite-index} and $e_{uCu^*}$ from \Cref{conjugate}).

 \end{remark}

\begin{remark}
  \begin{enumerate}
    \item In general, the dual conditional expectation of a tracial
      conditional expectation need not be tracial.

      For instance, consider the inclusion $B= \C \ni \lambda
      \hookrightarrow (\lambda, \lambda) \in A= \C \oplus \C$ with
      respect to the conditional expectation $E: A\to B$ given by
      $E((\lambda, \mu)) = \frac{\lambda + \mu}{2}$. Clearly, $E$ is a
      finite-index tracial conditional expectation and we see that one
      can identify $A_1$ with $M_2(\C)$ and then the dual conditional
      expectation $E_1: A_1 \to A$ is given by $E_1([a_{ij}]) =
      (a_{11}, a_{22})$. Clearly, $E_1(AB) \neq E_1(BA)$ for $A
      =\begin{pmatrix} 1 & 0 \\ 1 & 1 \end{pmatrix}$ and $B
      =\begin{pmatrix} 1 & 1 \\ 0 & 1 \end{pmatrix}$.

    \item      It is natural to  wonder whether traciality of $E$ can be dropped or not while
  showing that $uCu^*$ belongs $\mathrm{IMS (B,A,E)}$ with respect to
  $F_u$. And, it turns out that it can't be dropped always.

  For instance, consider $A= M_2(\C)$ and $B= \C I_2$ with the
  conditional expectation $E : A\to B$ given by \( E([a_{ij}]) =
  a_{11} t + a_{22} (1-t) \quad \text{where t}\neq 1/2 \) is
  fixed. Let $C= \{ \mathrm{diag}(\lambda, \mu) : \lambda,\mu \in\C\}$
  and $F: A\to C$ be the conditional expectation defined by
  $F([a_{ij}]) = \mathrm{diag}(a_{11},a_{22})$. Clearly, $E$ and $F$
  are finite-index conditional expectations with quasi-bases $\{ \sqrt {t} e_{11}, \sqrt{1-t} e_{12}, \sqrt{t} e_{21}, \sqrt{1-t} e_{22}\}$ and $\{ e_{ij} : 1 \leq i,j\leq 2 \}$, respectively, and
  $E_{\restriction_C}\circ F = E$. If $u = \begin{bmatrix}
    \frac{1}{\sqrt{2}} & \frac{i}{\sqrt{2}} \\ \frac{i}{\sqrt{2}} &
    \frac{1}{\sqrt{2}}
\end{bmatrix}$, then $u\in U(2)$ and 
\[
F_u ([a_{ij}]) = \begin{bmatrix}
(a_{11}+ a_{22})/2 & (a_{12}-a_{21})/2 \\ (a_{21}-a_{12})/2 & (a_{11} + a_{22})/2
\end{bmatrix}
\]
for all $[a_{ij}]\in A$. Thus, $E_{\restriction_{uCu^*}}\circ F_u
\big([a_{ij}]\big) = \frac{a_{11} + a_{22}}{2}$ which is not equal to
$E\big([a_{ij}]\big)$  (as $t \ne 1/2$).
  \end{enumerate}
\end{remark}

\section{Interior and exterior angles}

     Let $B \subset A$ be an inclusion of unital $C^*$-algebras with a
     conditional expectation $E: A \to B$. Then, for the $B$-valued
     inner product $\la \cdot, \cdot\ra_B$ on $A$ given by $\la x, y
     \ra_B=E(x^*y)$, one has the following well known
     analogue of the Cauchy-Schwarz inequality
     \begin{equation}\label{C-S-inequality}
\|\la x, y \ra_B\| \leq \|x\|_A \|y\|_A \text{ for all } x, y \in A,
     \end{equation}
     where $\|x\|_A:=\|E_B(x^*x)\|^{1/2}$. And, unlike for usual inner products, equality in (\ref{C-S-inequality}) does
not imply that $\{x, y\}$ is linearly dependent. For instance,
consider $B = \{\mathrm{diag}(\lambda, \mu) : \lambda, \mu \in \C\}$
in $A= M_2(\C)$ with the natural finite-index conditional expectation
$E: A \to B$ given by $E([x_{ij}])= \mathrm{diag}(x_{11},
x_{22})$. Then, for $ x = \mathrm{diag}(1,1)$ and $ y =
\mathrm{diag}(i,1)$ in $A$, one easily verifies that
     \(
\|\la x, y \ra_B\| = \|x\|_A \|y\|_A
\)
whereas $\{x, y\}$ is linearly independent.

Employing Inequality (\ref{C-S-inequality}), motivated by
\cite{BDLR}, Bakshi and the first named author introduced the
following definitions of the interior and exterior angles between
intermediate $C^*$-subalgebras.
\begin{definition}\cite{BG2} 
     Let $B \subset A$ be an inclusion of unital $C^*$-algebras with a
     finite-index conditional expectation $E: A \to B$. Then, for $C,
     D \in \mathrm{IMS}(B, A, E)\setminus\{B\}$, the interior angle between $C$ and
     $D$ (with respect to $E$), denoted as $\alpha(C, D)$,  
     is given by the expression
     \begin{equation}\label{angle-int}
   \cos(\alpha(C,D)) = \frac{\|\langle e_C - e_B, e_D - e_B
     \rangle_{A} \|}{\| e_C - e_B\|_{A_1} \| e_D -
     e_B\|_{A_1}};
     \end{equation}
     and, for $C, D \in \mathrm{IMS}(B, A, E)\setminus\{A\}$ with
     $C_1, D_1 \in \mathrm{IMS}(A, A_1, E_1)$, the exterior angle
     between $C$ and $D$ is defined as
    \begin{equation}\label{angle-ext}
   \beta(C,D) = \alpha(C_1, D_1),
    \end{equation}
   where $\alpha(C_1, D_1)$ is defined with respect to the dual
   conditional expectation $E_1: A_1 \to A$.

   By definition, both angles are allowed to take values only in
the interval   $[0,\frac{\pi}{2}]$.
\end{definition}

\begin{remark}
  \begin{enumerate}
\item Note that if $\mathrm{Ind}(E_{\restriction_D}),
  \mathrm{Ind}(E_{\restriction_C}) \in \mathcal{Z}(A)$, then by
  \Cref{eC-e1}(7), $C_1, D_1 \in \mathrm{IMS}(A, A_1, E_1)$. Thus,
  $\beta(C,D)$ is defined for such intermediate subalgebras.
\item If $B \subset C, D \subset A$ is a quadruple of simple unital
  $C^*$-algebras, then $\beta(C,D)$ is always defined.
  \end{enumerate}
  \end{remark}

We now derive some useful formulae for the interior and exterior
angles in terms of certain related quasi-bases.

\begin{proposition} \label{Int_QB}
 Let $B\subset A$ be an inclusion of unital $C^*$-algebras with a
 finite index conditional expectation $E:A\to B$ with quasi-basis $\{\lambda_i: 1 \leq i \leq p\}$. Let $C, D\in
 \mathrm{IMS} (B,A,E)\setminus\{B\}$ with respect to the conditional
 expectations $F:A \to C$ and $F':A \to D$, respectively. Let $\{ \mu_{j} : 1\leq
 \mu_{j} \leq m\}$ and $\{ \delta_{k}: 1\leq \delta_{k} \leq n\}$ be
 quasi-bases for $E_{\restriction_{C}}$ and $E_{\restriction_{D}}$,
 respectively. Then, we have the following:
 \begin{enumerate} 
 	\item The interior angle between $C$ and $D$ is given by
 	 \begin{equation}\label{int-expn}
 	 \cos(\alpha(C,D)) = \frac{\Vert (\mathrm{Ind}(E))^{-1} \left(
           \sum_{j,k} \mu_{j} E(\mu_{j}^* \delta_{k})\delta_{k}^* -
           1\right) \Vert}{\Vert
           (\mathrm{Ind}(E))^{-1}(\mathrm{Ind}(E_{\restriction_{C}})
             - 1) \Vert^{\frac{1}{2}} \Vert
             (\mathrm{Ind}(E))^{-1}(\mathrm{Ind}(E_{\restriction_{D}})
             - 1) \Vert^{\frac{1}{2}}}.
 \end{equation}
         In particular, if $\mathrm{Ind}(E)$ is a scalar, then 
\begin{equation}\label{int-expn-scalar}
	\cos (\alpha(C,D)) = \frac{\Vert \sum_{j,k} \mu_{j}
          E(\mu_{j}^* \delta_{k})\delta_{k}^* - 1\Vert} { \Vert
          {\mathrm{Ind}(E_{\restriction_{C}}) - 1} \Vert^{\frac{1}{2}} \Vert
          {\mathrm{Ind}(E_{\restriction_{D}}) - 1} \Vert^{\frac{1}{2}}}.
\end{equation}
\item Whenever $ \mathrm{Ind}(E_{\restriction_{C}})$ and
  $\mathrm{Ind}(E_{\restriction_{D}})$ belong to $\mcal Z(A)$, the
  exterior angle between $C$ and $D$ can be derived from
  (\ref{angle-ext}) using the following expressions:
    \begin{eqnarray*}\label{ext-expn}
    \lefteqn{\langle e_{C_{1}} - e_2 , e_{D_{1}}- e_2 \rangle_{A_{1}}}\\&  = &
 (\mathrm{Ind}(E_{1}))^{-1}
\Big[(\mathrm{Ind}(E_{\restriction_{C}}))^{-2}
  (\mathrm{Ind}(E_{\restriction_{D}}))^{-1} \sum_{i,i'} \lambda_i e_C
  \mathrm{Ind}(F')\\ & & \quad \times\, \sum_{j,k} \mu_{j} E(\mu_{j}^* \lambda_{i}^*
  \lambda_{i'} \delta_{k}) \delta_{k}^*) e_D \lambda_{i'}^* - 1
  \Big],
  \end{eqnarray*}
  \[
  \Vert e_{C_{1}} - e_2 \Vert_{A_2} = \Big\Vert
    (\mathrm{Ind}(E_{1}))^{-1} \Big[
    (\mathrm{Ind}(E_{\restriction_{C}}))^{-1} \big( \sum_{i}
    \lambda_{i} F(\mathrm{Ind}(F))e_C \lambda_{i}^*\big) - 1\Big]\Big\Vert^{\frac{1}{2}}
  \]
and a similar expression for  $\Vert e_{D_{1}} - e_2 \Vert_{A_2}$.
\end{enumerate}
\end{proposition}   
\begin{proof}
  (1) follows immediately by substituting the expressions for $e_C,
  e_D$, as  obtained in \Cref{eC-e1} (4),(5),(6),  in the definition of interior angle (\ref{angle-int}).\smallskip

 (2): Note that the dual conditional expectation $E_{1} : A_{1} \to A$
  is of finite index with a quasi-basis $\{ \lambda_{i}e_B
  (\mathrm{Ind}(E))^{1/2}\}$ - see
 \Cref{useful-obsns}(3). Further, from \Cref{eC-e1}(4), a quasi-basis for
  ${E_1}_{\restriction_{C_1}}$ is given by $\{ w_i := G(\lambda_{i} e_B
  (\mathrm{Ind}(E))^{1/2}) : 1 \le i \le n\}$, where $G: A_1 \to C_1$ is the conditional expectation as in the
  proof of \Cref{eC-e1}(7). Thus, by \Cref{eC-e1}(4), we see that
	\begin{eqnarray*}
	e_{C_1} & = & \sum_{i} w_i e_2 w_i^*\\ & = & \sum_{i}
        \lambda_{i} e_C (\mathrm{Ind}(E_{\restriction_{C}}))^{-1}
        (\mathrm{Ind}(E))^{1/2} e_2
        (\mathrm{Ind}(E))^{1/2}(\mathrm{Ind}(E_{\restriction_{C}}))^{-1}
        e_C \lambda_{i}^*,
	\end{eqnarray*}
        since $ \mathrm{Ind}(E_{\restriction_{C}}) \in \mathcal{Z}(A)
        \cap \mathcal{Z}(C)$ and $e_C \in C' \cap C_1$.
Thus,
\begin{eqnarray*}
\lefteqn{E_{2} (e_{C_1}) - E_{2}(e_2)}\\ & = &
(\mathrm{Ind}(E_1))^{-1} \left[ \left(\sum_{i} \lambda_{i} e_C
  (\mathrm{Ind}(E_{\restriction_{C}}))^{-1}
  (\mathrm{Ind}(E))^{1/2}(\mathrm{Ind}(E))^{1/2}(\mathrm{Ind}(E_{\restriction_{C}}))^{-1}
  e_C \lambda_{i}^*\right) - 1\right] \\ & = &
(\mathrm{Ind}(E_1))^{-1} \left[ \left(\sum_{i} \lambda_{i} e_C
  (\mathrm{Ind}(E_{\restriction_{C}}))^{-2} (\mathrm{Ind}(E)) e_C
  \lambda_{i}^*\right) - 1\right] \\ & = &
(\mathrm{Ind}(E_1))^{-1}\left[ \left(\sum_{i} \lambda_{i} e_C
  (\mathrm{Ind}(E_{\restriction_{C}}))^{-1} (\mathrm{Ind}(F)) e_C
  \lambda_{i}^*\right) - 1\right] \quad \text{(by
  Prop. \ref{eC-e1}(7a))} \\ & = & (\mathrm{Ind}(E_1))^{-1}\left[
  \left((\mathrm{Ind}(E_{\restriction_{C}}))^{-1}\sum_{i} \lambda_{i}
  F((\mathrm{Ind}(F)) e_C \lambda_{i}^*\right) -1 \right],
\end{eqnarray*}
which shows that
  \[
  \Vert e_{C_{1}} - e_2 \Vert_{A_2} = \Vert
    (\mathrm{Ind}(E_{1}))^{-1} \Big[
    (\mathrm{Ind}(E_{\restriction_{C}}))^{-1} \big( \sum_{i}
    \lambda_{i} F(\mathrm{Ind}(F))e_C \lambda_{i}^*\big) - 1 \Big] \Vert^{\frac{1}{2}}.
  \]

Further, as above, we
        have
        \[
e_{D_1}= \sum_{i}
        \lambda_{i} e_D (\mathrm{Ind}(E_{\restriction_{D}}))^{-1}
        (\mathrm{Ind}(E))^{1/2} e_2
        (\mathrm{Ind}(E))^{1/2}(\mathrm{Ind}(E_{\restriction_{D}}))^{-1}
        e_D \lambda_{i}^*;
        \]         \text{ so that}
\begin{eqnarray*}
\lefteqn{ e_{C_1}e_{D_1}}\\ & = & \sum_{i,i'}\lambda_{i} e_C
(\mathrm{Ind}(E_{\restriction_{C}}))^{-1}(\mathrm{Ind}(E))^{1/2}
E_{1}\left[(\mathrm{Ind}(E))^{1/2}(\mathrm{Ind}(E_{\restriction_{C}}))^{-1}e_C
\lambda_{i}^* \lambda_{i'} e_D
\right.\\ & & \qquad \times
\left. (\mathrm{Ind}(E_{\restriction_{D}}))^{-1}(\mathrm{Ind}(E))^{1/2}\right] e_2
(\mathrm{Ind}(E))^{1/2}(\mathrm{Ind}(E_{\restriction_{D}}))^{-1} e_D
\lambda_{i'}^* \\ & = & \sum_{i,i'} \lambda_{i} e_C
(\mathrm{Ind}(E_{\restriction_{C}}))^{-1}\, \mathrm{Ind}(F)\, E_{1}(e_C
\lambda_{i}^*
\lambda_{i'}e_D)(\mathrm{Ind}(E_{\restriction_{D}}))^{-1}
\mathrm{Ind}(F') e_2 e_D \lambda_{i'}^*,
\end{eqnarray*}
where the last equality holds because of  \Cref{eC-e1} (7)(a). Then, we obtain
\begin{eqnarray*}
\lefteqn{E_{2}(e_{C_1} e_{D_1}) - E_{2}(e_2)}\\ & = &
(\mathrm{Ind}(E_1))^{-1}\sum_{i,i'} \lambda_{i} e_C
(\mathrm{Ind}(E_{\restriction_{C}}))^{-1} (\mathrm{Ind}(F)) E_{1}(e_C
\lambda_{i}^*
\lambda_{i'}e_D)(\mathrm{Ind}(E_{\restriction_{D}}))^{-1}\\
 & &\quad \times\, \mathrm{Ind}(F') e_D \lambda_{i'}^* -(\mathrm{Ind}(E_1))^{-1}  \\
& =& (\mathrm{Ind}(E_{1}))^{-1}
\Big[(\mathrm{Ind}(E_{\restriction_{C}}))^{-2}
  (\mathrm{Ind}(E_{\restriction_{D}}))^{-1} \sum_{i,i'} \lambda_i e_C
  \mathrm{Ind}(F')\\ & & \quad \times\, \sum_{j,k} \mu_{j} E(\mu_{j}^* \lambda_{i}^*
  \lambda_{i'} \delta_{k}) \delta_{k}^*) e_D \lambda_{i'}^* - 1
  \Big].
\end{eqnarray*}
Since
\(
\langle e_{C_{1}} - e_2 , e_{D_{1}}- e_2 \rangle_{A_{1}} = 
E_{2}(e_{C_1} e_{D_1}) - E_{2}(e_2),
\)
we are done.\end{proof}

\begin{remark}
A priori, for $C, D \in \mathrm{IMS} (B,A,E)\setminus
      \{B\}$, it is not clear whether $\alpha(C,D)=0$ implies $C=D$ or
      not. Though, when $B \subset A$ is an irreducible inclusion of
      simple unital $C^*$-algebras, then it is known to be true - see
      \cite[Proposition 5.10]{BG2}. Also, this phenomenon holds for a
      certain collection of intermediate subalgebras even in some
      non-irreducible set up as we shall see ahead in \Cref{Hadamard}.
\end{remark}

\section{Possible values of the interior angle}\label{possible-values}

Throughout this section, we let $A= M_2(\C)$, $ B=
\C I_2$, $\Delta=\{\mathrm{diag}(\lambda, \mu) : \lambda, \mu \in \C\}$,
$E: A \to B$ denote the  canonical (tracial) conditional expectation  given by
\[
E([a_{ij}]) =
\frac{(a_{11}+a_{22})}{2}I_2 \text{ for } [a_{ij}]\in A
\]
 and
$F: A \to \Delta$ denote the conditional expectation given by $F([a_{ij}])
= \mathrm{diag}(a_{11}, a_{22})$.

The following useful observations are standard - see \cite[Example
  2.4.5]{Wat}.
\begin{lemma}\label{C-M2-basic-construction}
With running notations, the following hold: \begin{enumerate}
\item $E$ is a finite-index conditional expectation with a quasi-basis
\[
\{\sqrt{2}e_{ij} : 1 \leq i,j \leq 2\},
\]
where $\{e_{ij}: 1 \leq i, j \leq 2\}$ denotes the set of standard
matrix units of $M_2(\C)$.
\item \( \mathrm{Ind}(E) =  4 \) and  $E$ is the (unique) minimal conditional expectation from $A$ onto $B$.
\item The $C^*$-basic construction $A_1$ for $B \subset A$, with respect to the
  conditional expectation $E$, can be identified with $M_4(\C)$ and the
  Jones projection $e_1$ corresponding to the conditional expectation
  $E$ is given by
  \[
  e_1 =  \frac{1}{2}
  \begin{bmatrix}
    1 & 0 & 0 & 1\\
    0 & 0 & 0 & 0 \\
    0 & 0 & 0 & 0\\
    1 & 0 & 0 & 1
  	\end{bmatrix}.
  \]
\item Identifying $M_4(\C)$ with $M_2(\C) \ot M_2(\C)$, the dual conditional
  expectation $E_1 : A_1 \to A $ is given by $E_1 = \mathrm{id_{M_2}}
  \ot E$; thus,
  \[
  E_1\left(X\right) = \begin{bmatrix}
	E(X_{(1,1)}) & E(X_{(1,2)}) \\ E(X_{(2,1)}) & E(X_{(2,2)})
  \end{bmatrix}, X \in M_4(\C),
  \]
  where $X_{(i,j)}$ denotes the $(i,j)^{\text{th}}$ $2 \times 2$ block
  of a matrix $X \in M_4(\C)$.
\end{enumerate}
\end{lemma}

\begin{lemma}\label{E-F-finite-index}
  \begin{enumerate}
   \item $F$ has finite index with a  quasi-basis
     $\{ e_{ij} : 1\leq i,j \leq 2\}$  and scalar index equal to $2$. \smallskip
     \item $\Delta \in \mathrm{IMS}(B, A, E)$ with
      respect to the conditional expectation $F$ and the corresponding Jones projection in $\Delta_1$ ($\subset A_1 =M_4(\C)$) is given by
      \[
      e_\Delta = \begin{bmatrix}
  	1 & 0 & 0 & 0\\ 0 & 0 & 0 & 0\\ 0 & 0 & 0 & 0\\ 0 & 0 & 0 & 1 
      \end{bmatrix}.
      \]
  \end{enumerate}
\end{lemma}

\begin{lemma}\label{conjugate}
For every unitary $u$ in $A$,
  \begin{enumerate}
      \item    the map
   $F_u: A
\to u \Delta u^*$ given by
\(
F_u= \mathrm{Ad}_{u} \circ F \circ
\mathrm{Ad}_{u^*}
\) is a finite-index conditional expectation with a quasi-basis
\(
\{ e_{ij}u : 1\leq i,j \leq 2\}\)
and $\mathrm{Ind}(F_u) = 2 $; 
\item $D:=u\Delta u^* \in \mathrm{IMS}(B, A, E)$ with respect to the
  conditional expectation $F_u$; and,
  \item if $ u = [\lambda_{ij}] $, then
  the corresponding Jones projection in $D_1$ $(\subset A_1=M_4(\C))$ is given by
  $e_D=[x_{ij}]$, where
\[
     x_{11}   =  |\lambda_{11}|^4 + |\lambda_{12}|^4
     , x_{12}  =  \lambda_{21}
     \bar{\lambda}_{11}(|\lambda_{11}|^2 - |\lambda_{12}|^2),
     x_{14}      =  2|\lambda_{11}|^2 |\lambda_{12}|^2,
     \]
     \[x_{22} =
        2|\lambda_{11}|^2 |\lambda_{21}|^2,
     x_{23}  =   2 \bar{\lambda}_{21}^2 \lambda_{11}^2
   \]
and the remaining entries are given by $x_{12} = \bar{x_{13}} = \bar{x_{21}} = {x_{31}} = -\bar {x_{24}} = -
     x_{42}$, $x_{41} =  x_{14} $, $x_{33} = x_{22}$, $x_{32} = \bar{x_{23}}$ and 
     $x_{44} =     x_{11}$.
   \end{enumerate}
  \end{lemma}

\begin{proof}
(1): Clearly, the map $F_u: A \to D$
is a conditional expectation and we can easily verify that 
\[
x = \sum_{i,j} e_{ij}u F_u(u^*e_{ij}^*x) = \sum_{i,j} F_u(x
e_{ij}u) u^* e_{ij}^*
\]
for all $x
\in A$. Thus, $\{e_{ij}u : 1 \leq i,j \leq 2\}$ is a quasi-basis for $F_u$ and $\mathrm{Ind} (F_u) = 2= \mathrm{Ind}(F)$. \smallskip

 Since $E$ satisfies the tracial property and
$\mathcal{N}_A(B) =\mathcal{A}$, (2) follows from \Cref{conj_F}.\smallskip

(3): After some routine calculation, for any $\begin{bmatrix}a & b \\ c
  & d \end{bmatrix} \in A$, we obtain
\[
F_u \left( \begin{bmatrix}
	a & b \\ c & d
\end{bmatrix} \right )  =  \begin{bmatrix}
x |\lambda_{11}|^2 + y | \lambda_{12} |^2 & x
\bar{\lambda}_{21}\lambda_{11} + y \bar{\lambda}_{22} \lambda_{12}
\\ x \lambda_{21} \bar{\lambda}_{11} + y \lambda_{22}
\bar{\lambda}_{12} & x |\lambda_{12}|^2 + y |\lambda_{11} |^2
\end{bmatrix},
\]
where $ x = a |\lambda_{11}|^2 + d |\lambda_{12}|^2 + b \lambda_{21}
\bar{\lambda}_{11} + c \bar{\lambda}_{21} \lambda_{11} $ and $ y = a
|\lambda_{12}|^2 + d |\lambda_{11}|^2 + b \lambda_{22}
\bar{\lambda}_{12} + c \bar{\lambda}_{22} \lambda_{12} $. Since $u$ is a unitary, we have $
\bar{\lambda}_{12} \lambda_{22} = -(\bar{\lambda}_{11}\lambda_{21})$,
$ \lambda_{12}\bar{\lambda}_{22} = -(\lambda_{11} \bar{\lambda}_{21})
$, $|\lambda_{11}|^2 = |\lambda_{22}|^2$ and $ |\lambda_{12}|^2 =
|\lambda_{21}|^2 $; thus, we further deduce that
\begin{eqnarray*}
 x |\lambda_{11}|^2 + y |\lambda_{12}|^2 & = & a (
 |\lambda_{11}|^4 + | \lambda_{12}|^4) + 2 d |\lambda_{11}|^2
 |\lambda_{12}|^2 + c \bar{\lambda}_{21} \lambda_{11}(|\lambda_{11}|^2
 - |\lambda_{12}|^2) \\ & & \qquad + b\lambda_{21}
 \bar{\lambda}_{11}(|\lambda_{11}|^2 -|\lambda_{12}|^2); \text{ and }\\
  x |\lambda_{12}|^2 + y |\lambda_{11}|^2 
  &  = & 2a |\lambda_{11}|^2 |\lambda_{12}|^2 + d(|\lambda_{11}|^4 + |\lambda_{12}|^4) + b\lambda_{21} \bar{\lambda}_{11}(|\lambda_{12}|^2 - |\lambda_{11}|^2)\\
  & & \qquad + c \bar{\lambda}_{21} \lambda_{11}(|\lambda_{12}|^2 - |\lambda_{11}|^2). 
\end{eqnarray*}
Then, using the above expression for $F_u([a_{ij}])$ for $[a_{ij}] \in A$ and the
matrix $e_D$ as given in the statement, we can easily verify that 
\begin{enumerate}[(i)]
\item $e_D x e_D = F_u(x)e_D$ and
\item $e_D(x) = F_u(x)$
\end{enumerate}
for all $ x \in A$. This completes the proof.
\end{proof}

We are now all set to derive a concrete expression for the interior
angle between $\Delta$ and its conjugate $u\Delta u^*$, in terms of the entries of $u$.
\begin{theorem}\label{angle-CD}
  If $u = [ \lambda_{ij}]\in U(2)$, then
\begin{equation}
\cos(\alpha(\Delta, u \Delta u^*)) = \sqrt{1-(2 |\lambda_{11}| |\lambda_{12}|)^4}.
\end{equation}
  \end{theorem}

\begin{proof} Let $D = u \Delta u^*$. From the matrix expressions of  $e_1$, $e_\Delta$ and $e_D$ obtained above, we easily see that
$E_1(e_1) = \frac{1}{4} I_2$, $E_1(e_\Delta)= \frac{1}{2}I_2$, $E_1(e_D) = 
  \frac{1}{2} I_2$ and
  \[
  E_1(e_\Delta e_D) = \begin{bmatrix}
   {(|\lambda_{11}|^4 + |\lambda_{12}|^4)}/{2} & \bar{\lambda}_{21}
    \lambda_{11}(|\lambda_{11}|^2 - |\lambda_{12}|^2)/2
    \\ \lambda_{21} \bar{\lambda}_{11}(|\lambda_{12}|^2 -
    |\lambda_{11}|^2)/2 & (|\lambda_{11}|^4 + |\lambda_{12}|^4)/2
  \end{bmatrix}.
  \]
  Thus,
  \[
\|e_\Delta -e_1 \|_{A_1} = \| E(e_\Delta - e_1)\|= \frac{1}{2} = \| E(e_D - e_1)\| = \|e_{D} -e_1 \|_{A_1}. 
  \]
  Next, we calculate $\|E_1(e_\Delta e_D - e_1) \|$.   Let
  \[
  T= E_1(e_\Delta e_D - e_1) = \begin{bmatrix} (|\lambda_{11}|^4 +
    |\lambda_{12}|^4)/2 - 1/4 & \bar{\lambda}_{21}
    \lambda_{11}(|\lambda_{11}|^2 - |\lambda_{12}|^2)/2
    \\ \lambda_{21} \bar{\lambda}_{11}(|\lambda_{12}|^2 - |\lambda_{11}|^2)/2 &
    (|\lambda_{11}|^4 + |\lambda_{12}|^4)/2 - 1/4
  \end{bmatrix}.
  \]
  Note that, $T^*T$ turns out to be a scalar  matrix with eigen value
  $\lambda$, where
	\begin{eqnarray*}
		\lambda & = & \left(\frac{|\lambda_{11}|^4 +
                  |\lambda_{12}|^4}{2} - \frac{1}{4}\right)^2 +
                \frac{|\lambda_{11}|^2 |\lambda_{21}|^2
                  (|\lambda_{12}|^2 - |\lambda_{11}|^2)^2}{4} \\ & = & \frac{1}{16} \left( 2(|\lambda_{11}|^4 + |\lambda_{12}|^4) - 1\right) ^2 + \frac{|\lambda_{11}|^2 |\lambda_{21}|^2
                  (|\lambda_{12}|^2 - |\lambda_{11}|^2)^2}{4} \\ & = & \frac{1}{16} \left( 2(|\lambda_{11}|^4 + |\lambda_{12}|^4) - (|\lambda_{11}|^2 + |\lambda_{12}|^2)^2\right) ^2 + \frac{|\lambda_{11}|^2 |\lambda_{12}|^2
                  (|\lambda_{12}|^2 - |\lambda_{11}|^2)^2}{4} \\
                  & & \qquad \qquad
                \qquad \qquad \qquad \text{ (since
                  $|\lambda_{11}|^2+|\lambda_{12}|^2 = 1$ and
                  $|\lambda_{21}| = |\lambda_{12}|$)}\\
                   & = &
                \left(\frac{|\lambda_{11}|^2 -
                  |\lambda_{12}|^2}{4}\right)^2 +
                \frac{|\lambda_{11}|^2 |\lambda_{12}|^2
                  (|\lambda_{12}|^2 - |\lambda_{11}|^2)^2}{4}\\
 & = &                \left(\frac{(|\lambda_{11}|^2 -
                  |\lambda_{12}|^2)}{4}\right)^2 \left(1+ 4
                |\lambda_{11}|^2 |\lambda_{12}|^2\right)\\ & = &
                \frac{1}{16}\left( (|\lambda_{11}|^2 +
                |\lambda_{12}|^2)^2 - 4|\lambda_{11}|^2
                |\lambda_{12}|^2\right) \left( (1+4|\lambda_{11}|^2
                |\lambda_{12}|^2)\right)\\
                & = & \frac{1}{16} \left(1-(2 |\lambda_{11}||\lambda_{12}|)^4\right).
\end{eqnarray*}
        Thus,
        \[
         \Vert \langle e_\Delta - e_1 \; e_D - e_1 \rangle_{A}\Vert = \Vert E_1(e_\Delta e_D - e_1)\Vert 
   =  \Vert T \Vert
  =  \left(\sqrt{1- (2|\lambda_{11}| |\lambda_{12}|)^4}\right)/4.
        \]
Finally, substituting the above values of $\|e_\Delta -e_1 \|_{A_1}$,
$\|e_{D} -e_1 \|_{A_1}$ and $ \Vert \langle e_\Delta - e_1 \; e_D - e_1
\rangle_{A}\Vert $ in  (\ref{angle-int}), we obtain
\[
\cos(\alpha(\Delta, u\Delta u^*)) = \sqrt{1-(2 |\lambda_{11}| |\lambda_{12}|)^4}.
\]
  \end{proof}

Recall that a unitary matrix whose all entries have the same modulus is
called a {\em Hadamard matrix}. Also, if $(B,C,D,A)$ is a quadruple of
finite von Neumann algebras (i.e., $B \subset C, D \subset A$) with a
faithful normal tracial state $\tau:A \to \C$, then $(B,C,D,A)$ is
said to be a commuting square if $E^A_C E^A_D = E^A_B = E^A_D E^A_C$,
where $E^A_X: A \to X$ denotes the unique $\tau$-preserving conditional
expectation from $A$ onto any von Neumann subalgebra $X$ of $A$.

\begin{corollary}\label{Hadamard}  Let $u=\in U(2)$. Then,
  \begin{enumerate}
    \item $\alpha(\Delta, u \Delta u^*) = \frac{\pi}{2}$ if and only
      if $u$ is a Hadamard matrix if and only if $(B, \Delta, u^*
      \Delta u, A)$ is a commuting square; and,
    \item if $u=[\lambda_{ij}]$, then $\alpha(\Delta, u \Delta u^*) =
      0 $ if and only if either $u\in \Delta$ or $\lambda_{11} = 0
      =\lambda_{22}$.\\ In particular, $\alpha(\Delta, u\Delta u^*)=0$
      if and only if $\Delta = u\Delta u^*$.
  \end{enumerate}

\end{corollary}
\begin{proof}
  (1): From \Cref{angle-CD}, we observe that 
  \begin{eqnarray*}
\cos(\alpha(\Delta,  u\Delta u^*)) = 0 & \Leftrightarrow &
\sqrt{1-(2|\lambda_{11}||\lambda_{12}|)^4} = 0 \\ & \Leftrightarrow &
|\lambda_{11}||\lambda_{12}| = \frac{1}{2} \\ & \Leftrightarrow &
|\lambda_{11}| = |\lambda_{12}| \qquad \text{(since $|\lambda_{11}|^2
  + |\lambda_{12}|^2$ = 1)} \\ & \Leftrightarrow &
|\lambda_{11}|= |\lambda_{12}| = |\lambda_{21}| = |\lambda_{22}|
\\ &\Leftrightarrow & \text{ $u$ is Hadamard Matrix}.
  \end{eqnarray*}
Note that $F: M_2 \to \Delta$ and $F_u: M_2 \to u\Delta u^*$ are the
unique trace preserving conditional expectations. And, it is a well
known fact - see, for instance, \cite[$\S 5.2.2$]{JS} - that $(B, \Delta, u^*
\Delta u, A)$ is a commuting square if an only if $u$ is a Hadamard
matrix.  \smallskip

  (2): Again, from \Cref{angle-CD},  \begin{eqnarray*}
           \cos(\alpha(\Delta, u \Delta u^*)) = 1 
          & \Leftrightarrow & |\lambda_{11}||\lambda_{12}|= 0  \\ 
          & \Leftrightarrow & |\lambda_{11}| = 0 \quad\text{or} \quad |\lambda_{12}| = 0 \\
          & \Leftrightarrow & |\lambda_{11}| = 0 = |\lambda_{22}| \quad\text{or} \quad \text{ $u$ is diagonal}. \qquad \text{(as $u$ is unitary)}
\end{eqnarray*}
  \end{proof}

We can now deduce our assertion that the interior angle attains all
values in $[0, \pi/2]$.
\begin{corollary} 
  \[
  \{\alpha(\Delta,u\Delta u^*): u \in U(2)\} = \left[0, \frac{\pi}{2}\right].
  \]
\end{corollary}
\begin{proof}
Note that, for each $u=[\lambda_{ij}]\in U(2)$,
$0 \leq (2|\lambda_{11}||\lambda_{12}|)^4 \leq 1$. Thus, we can define a
map $\varphi : U(2) \to [0,1]$ given by
\[
		\varphi \left([\lambda_{ij}]\right) =
                \sqrt{1-(2|\lambda_{11}||\lambda_{12}|)^4}.
\]
Clearly, $\varphi$ is a continuous function. Since $U(2)$ is
connected, it follows that $\varphi(U(2))$ is also connected. Note
that, from \Cref{Hadamard}, we have $\varphi(u)= 0$ for any complex
Hadamard matrix $u\in U(2)$ and $\varphi(I_2)=1$. Hence, $ \varphi(U(2)) =
[0,1] $. In particular, in view of \Cref{angle-CD}, we obtain
\[
  \{\alpha(\Delta,u\Delta u^*): u \in U(2)\} = \left[0,\frac{\pi}{2}\right],
  \]
as was desired.	\end{proof}

\begin{corollary}
  There exist $C, D \in \mathrm{IMS}(B, A, E)$ such that $e_C
  e_D \neq e_D e_C$.
\end{corollary}

\begin{proof}
Fix a $u =[\lambda_{ij}]\in U(2)$ and let $C =
\Delta$ and $D = u \Delta u^*$.   Then, both $C , D \in
\mathrm{IMS}(A,B,E)$ and using the values of $e_C $ and $e_D$ from
Lemmas \ref{E-F-finite-index} and \ref{conjugate}, we obtain
\[ e_C e_D= \begin{bmatrix} |\lambda_{11}|^4 + |\lambda_{12}|^4 &
  \lambda_{21} \bar{\lambda}_{11}(|\lambda_{11}|^2 - |\lambda_{12}|^2)
  & \bar{\lambda}_{21} \lambda_{11}(|\lambda_{11}|^2 -
  |\lambda_{12}|^2) & 2|\lambda_{11}|^2 |\lambda_{12}|^2 \\ 0 & 0 & 0
  & 0\\ 0& 0 & 0 & 0\\ 2 |\lambda_{11}|^2 |\lambda_{12}|^2 &
  \lambda_{21} \bar{\lambda}_{11}(|\lambda_{12}|^2 - |\lambda_{11}|^2)
  & \bar{\lambda}_{21} \lambda_{11} (|\lambda_{12}|^2 -
  |\lambda_{11}|^2) & |\lambda_{11}|^4 + |\lambda_{12}|^4
\end{bmatrix}
\]
 and 
 \[
 e_D e_C = \begin{bmatrix} |\lambda_{11}|^4 + |\lambda_{12}|^4 & 0 & 0
   & 2|\lambda_{11}|^2 |\lambda_{12}|^2 \\ \bar{\lambda}_{21}
   \lambda_{11}(|\lambda_{11}|^2-|\lambda_{12}|^2) & 0 & 0 &
   \bar{\lambda}_{21} \lambda_{11}(|\lambda_{12}|^2 -
   |\lambda_{11}|^2)\\ \lambda_{21}
   \bar{\lambda}_{11}(|\lambda_{11}|^2 - |\lambda_{12}|^2) & 0 & 0 &
   \lambda_{21} \bar{\lambda}_{11}(|\lambda_{12}|^2 -
   |\lambda_{11}|^2)\\ 2 |\lambda_{11}|^2 |\lambda_{12}|^2 & 0 & 0 &
   |\lambda_{11}|^4 + |\lambda_{12}|^4
 \end{bmatrix}.
 \]
Thus, if $u$ is neither a diagonal matrix nor a Hadamard matrix nor
$\lambda_{11} = 0 = \lambda_{22}$, then we see that $C \neq D$ and $
e_C e_D \neq e_D e_C$.
\end{proof}

  \section{Angles between intermediate crossed product subalgebras of crossed product inclusions}\label{crossed-products}

Recall that if a countable discrete group $G$ acts on a unital
$C^*$-algebra $P$ via a map $\alpha: G \to \mathrm{Aut}(P)$, then the
space $C_c(G,P)$ consisting of compactly supported $P$-valued
functions on $G$ can be identified with the space $ \{
\sum_{\text{finite}}a_g g: a_g \in P, g \in G\}$ of formal finite sums
and is a unital $*$-algebra with respect to (the so called twisted)
multiplication given by the convolution operation
       \[
\left(\sum_{s\in I} a_s s\right) \left(\sum_{t\in J}
       b_t t\right) = \sum_{s \in I, t\in J} a_s \alpha_s(b_t)st
       \]
            and involution given by
       \[
\left(\sum_{s\in I} a_s s\right)^*=\sum_{s\in I} \alpha_{s^{-1}}(a_s^*)s^{-1} 
       \]
  for any two finite sets $I$ and $J$ in $G$. Further, the reduced crossed
  product $P \rtimes_{\alpha, r} G$ and the universal crossed product
  $P \rtimes_{\alpha} G$ are defined, respectively, as the completions
  of $C_c(G,P)$ with respect to the reduced norm
  \[
  \left \| \sum_{\text{finite}} a_g g \right \|_r:=\left \| \sum_{\text{finite}}\pi(a_g) (1 \otimes \lambda_g) \right \|_{B(H \otimes \ell^2(G))},
  \]
where $P \subset B(H)$ is a (equivalently, any) fixed faithful
representation of $A$, $\lambda: G \to B(\ell^2(G))$ is the left
regular representation and $\pi: A \to B(H\otimes \ell^2(G))$ is the
representation satisfying $\pi(a)(\xi \otimes \delta_g) =
\alpha_{g^{-1}}(a)(\xi) \otimes \delta_g$ for all $\xi \in H$ and $g
\in G$; and the universal norm
\[
\| x  \|_u:= \sup_\pi
\|\pi(x)\|\text { for } x \in C_c(G, P),
\]
where the supremum runs over all (cyclic) $*$-homomorphisms
  $\pi: C_c(G,P) \to B(H)$, respectively. We suggest the reader to see \cite{BO, wil} for more on crossed products.

When $G$ is a finite group, then it is well known that the  reduced and universal norms coincide on $C_c(G,P)$ and $C_c(G,P)$
is complete with the common norm; thus, 
$P\rtimes_{\alpha, r} G = C(G,P) = P\rtimes_{\alpha} G$ (as $*$-algebras).
  
  In this section, analogous to \cite[Proposition 2.7]{BG1}, we
  derived a concrete value for the interior angle between intermediate
  crossed product subalgebras of an inclusion of crossed product
  algebras.

The following important observations are well known.
\begin{proposition}\cite{cho,MK}\label{ch-cg-basics}
          Let $G$ be a countable discrete group and $H$ be its
          subgroup.  Let $P$ be a unital
          $C^*$-algebra such that $G$ acts on $P$ via a map $\alpha: G
          \to \mathrm{Aut}(P)$.  Let $A:=P \rtimes_{\alpha, r} G$
          (resp., $A:= P\rtimes G$) and $B:= P \rtimes_{\alpha, r} H$
          (resp., $B:= P\rtimes H$). Then,
 \begin{enumerate}
 \item  the canonical injective $*$-homomorphism
   \[
C_c(H, P)\ni \sum_{\mathrm{finite}} a_h h \mapsto \sum_{\mathrm{finite}} a_h h \in  C_c(G, P)
\]
extends to an injective $*$-homomorphism from $B$ into $A$; and
\item the natural map
  \[
C_c (G, P) \ni \sum_{\mathrm{finite}} a_g g \mapsto \sum a_h h \in C_c(H, P)
  \]
  extends to a conditional expectation $E: A \to B$.

  Moreover, $E$ has finite index if and only if $[G:H]< \infty$ and in that case a 
 quasi-basis for $E$ is given by $\{g_i : 1 \leq i \leq [G:H]\}$ for any  set $\{g_i\}$ of left coset
 representatives of $H$ in $G$ and $E$ has scalar index equal to $[G:H]$.
 \end{enumerate}
\end{proposition}
    \begin{proof}
    	(1) follows from \cite{cho} (also see \cite[Remark 3.2]{MK}) and the first part of the proof
      of \cite[Proposition 3.1]{MK}, respectively.\smallskip

      (2): Consider the canonical $C_c(H,P)$-bilinear projection $ E_0 : C_c(G,P)\to C_c(H,P) $ given by
        \[
        	E_0 \left(\sum_{\text{finite}} a_g g\right) = \sum a_h h .
       \]
Then, from \cite{cho} (also see \cite[Remark 3.2]{MK}) and
\cite[Proposition 3.1, Remark 3.2]{MK}, respectively, it follows that $E_0$ extends
to a conditional expectation from $A$ onto $B$. Also, from
\cite[Theorem 3.4]{MK}, it follows that $E$ has finite index (with a
quasi-basis as in the statement) if and only $[G:H] < \infty$.    \end{proof}

        \begin{proposition}\label{cross-prod-angle}
 Let $G, H, P, \alpha, A, B$ and $E$ be as in \Cref{ch-cg-basics} with
 $[G:H] < \infty$ and $\{ g_i: 1 \leq i \leq [G:H]\}$ be a set of left
 coset representatives of $H$ in $G$. Let $K$ and $L$ be proper intermediate
 subgroups of $H \subset G$ and let $C:=P
 \rtimes_{\alpha, r} K$ (resp., $ P \rtimes_\alpha K$) and $D:=P
 \rtimes_{\alpha, r} L$ (resp., $P \rtimes_\alpha L$). Then, $C, D \in
 \mathrm{IMS}(B, A, E)\setminus\{A,B\}$ and the interior angle between them is given by
  \[
\cos(\alpha(C,D)) = \frac{[K\cap L: H] - 1}{\sqrt{[K:H]-1}\sqrt{[L:H]-1}}.
\]
        \end{proposition}

    \begin{proof}
      Note that, $B \subset C, D \subset A$, by
      \Cref{ch-cg-basics}. Also, $[G:K]$ and $[G:L]$ are both finite as
      $[G:H]$ is finite. So, $C, D \in \mathrm{IMS(B,A,E)}$ with
      respect to the natural finite-index conditional expectations
      guaranteed by \Cref{ch-cg-basics}.
     
      Fix left coset representatives $\{k_r: 1 \leq r \leq [K:H]\}$
      and $\{ {l_s} : 1 \leq s \leq [L:H]\}$ of $H$ in $K$ and $L$,
      respectively.  Then, it is readily seen that $E_{\restriction_C}
      : C \to B$ and $E_{\restriction_D} : D \to B$ have quasi-bases
      $\{ {k_r} : 1\leq r \leq [K:H]\}$ and $\{ {l_s} : 1 \leq s
      \leq [L:H]\}$, respectively.  Then,  from  (\ref{int-expn-scalar}), we obtain 
\begin{eqnarray*}
	\cos(\alpha(C,D)) &=& \frac { \Vert (\sum_{r,s} k_r E(k_{r}^*
          l_s)l_{s}^*)-1 \Vert}{ \Vert \sqrt{[K:H]-1}\Vert
          \Vert\sqrt{[L:H] - 1} \Vert} \\ & = & \frac {\Vert
          \left(\sum_{\{r,s : (k_rH)\cap (l_sH) \ne \phi\}} k_r
          k_{r}^* l_s l_{s}^* \right) - 1 \Vert} { \sqrt{[K:H]-1}
          \sqrt{[L:H] - 1}} \\ & = & \frac{[(K \cap L):H] -1
        }{\sqrt{[K:H]-1} \sqrt{[L:H] - 1}},
\end{eqnarray*}
where the last equality holds  because the map 
\[
\{(r,s) : k_rH\cap l_sH \neq \emptyset\} \ni (r,s) \mapsto k_rH = l_sH \in (K\cap L)/H
\]
is a bijection.
    \end{proof}

        \begin{corollary}\label{Coro_4.3}
          Let the notations be as in  \Cref{cross-prod-angle}. Then,
          \begin{enumerate}
          \item $\alpha(C, D)= \frac{\pi}{2}$ if and only if $K \cap L = H$; and,
          \item $\alpha (C, D) =
            0$ if and only if $K = L$.
          \end{enumerate}
 In particular, if $C_g:=P \rtimes_{\alpha}(g^{-1}Kg)$ (resp., $P \rtimes_{\alpha,r}(g^{-1}Kg) $)
 then $\alpha(C,C_g)=0$ for all $g \in G$ if and only if $K$ is normal in
 $G$.
 \end{corollary}
        \begin{proof}
          (1) is straight forward and, for (2), we just need to show the necessity.

          Note that $\alpha(C,D)=0$
          implies that $ \frac{[(K \cap L) : H]-1}{\sqrt{[K:H] -1}
            \sqrt{[L:H] - 1}} =1$, which then implies that
          \[
          \left(\sqrt{\frac{[(K \cap L) : H]-1}{[K:H] -1}} \right)
\left(\sqrt{\frac{[(K \cap L) : H]-1}{[L:H] -1}} \right)           = 1
          \]
          and
          that $[K \cap L : H] \neq 1$. Note that
          \[
0<          \frac{{[(K \cap L) : H]-1}}{{[K:H] -1}}, \frac{{[(K \cap L) : H]-1}}{{[L:H] -1}}  \leq 1;
          \]
          so, it follows that 
          \[
\frac{{[(K \cap L) : H]-1}}{{[K:H] -1}} = 1= \frac{{[(K \cap L) : H]-1}}{{[L:H] -1}}. 
          \]
Hence, $K = K \cap L  = L$.
                \end{proof}

        Recall that for a subgroup $H$ of a group $G$, its normalizer
        is given by
        \[
\mathcal{N}_G(H)=\{g\in G: g^{-1}Hg=H\}.
        \]
        \begin{corollary}
          Let $G, H$ and $K$ be as in \Cref{cross-prod-angle}. If $g \in
          \mathcal{N}_G(H)$,  then
          $\alpha(C,C_g)= 0$ if and only if
          $g \in \mathcal{N}_G(K)$, where $C_g$ is same as in \Cref{Coro_4.3}
        \end{corollary}
 
        \begin{proof} Let $L:=g^{-1}Kg$.
Since $[K:H] = [L:H] $, from Expression (\ref{angle-CrG}), we obtain
\[
\cos(\alpha(C,C_g)) = \frac{[(K\cap
    L):H]-1}{[K:H]-1}.
\]
Thus, $\alpha(C, C_g))=0$ if and only if $K \cap (g^{-1}Kg)
= K$ if and only if $g \in \mcal{N}_G(K)$.
          \end{proof}

    Note that, if $P=\C$ and $\alpha: G \to \mathrm{Aut}(\C)$ is the
        trivial representation, then we know that $C^*_r(G) = \C
        \rtimes_{\alpha, r} G$ and $C^*(G) = \C \rtimes_{\alpha}
        G$. Thus, we readily deduced the following:

        \begin{corollary}\label{angle-gp-algebras}
          Let $G$ be a countable discrete group with subgroups $H, K$
          and $L$ such that $H \subseteq K \cap L$, $H \neq K , L$ and
          $[G:H]< \infty$.  Let $A:=C^*_r(G)$ (resp., $C^*(G)$),
          $B:=C^*_r(H)$ (resp., $C^*(H)$), $C:=C^*_r(K)$ (resp.,
          $C^*(K)$) and $D:=C^*_r(L)$ (resp., $C^*(L)$). Then, $C, D
          \in \mathrm{IMS} (B, A, E)$ and
   \begin{equation}\label{angle-CrG}
\cos(\alpha(C,D)) = \frac{[K\cap L: H] - 1}{\sqrt{[K:H]-1}\sqrt{[L:H]-1}},
  \end{equation}
where $E: A \to B$ is the conditional expectation as in \Cref{ch-cg-basics} with $P=\C$.  
\end{corollary}

\begin{example} \label{example-gps}
   Let $G= \Z_3 \oplus \Z_3 \oplus \Z_5 \oplus \Z_5 $. Consider its 
    subgroups $ K= \Z_3 \oplus \Z_3 \oplus \Z_5 \oplus (0)$ , $ L= (0)
    \oplus \Z_3 \oplus \Z_5 \oplus (0)$ and $ H= (0) \oplus (0)\oplus
    \Z_5 \oplus (0)$. Then, 
  \[
  \cos\left( \alpha (\C[K],\C[L])\right) = \frac{1}{2}.
  \] 
  Thus,  $\alpha(\C[K], \C[L])$ = $ \pi/3$.

  In particular, this illustrates that if $B
  \neq C \subsetneq D \subsetneq A$, then $\alpha (C,D)$ need not be
  $0$.
  \end{example}

\end{document}